\documentclass[12 pt, leqno]{amsart}

\usepackage{graphicx}
\usepackage[colorlinks,citecolor=red,pagebackref,hypertexnames=false]{hyperref}
\usepackage{amsfonts,slantsc}
\usepackage{amsmath,amssymb,amsthm,amssymb,amscd,cancel,stackrel}
\usepackage{blkarray}
\usepackage{multirow, mathtools,leftidx}
\usepackage{memhfixc}
\usepackage{latexsym}
\usepackage{tikz}
\usepackage{color}
\usepackage{enumerate}
\usepackage{relsize}

\newcommand*{\linkby[1]}{\stackrel{#1}{\thicksim}{}}

\DeclareMathOperator{\Char}{char}

\DeclareMathOperator{\Hf}{\mathcal{H}}

\DeclareMathOperator{\Span}{Span}
\DeclareMathOperator{\Tor}{Tor}

\DeclareMathOperator{\EGH}{EGH} 

\newcommand{\floor}[1]{\lfloor #1 \rfloor}

\newtheorem{theorem}{Theorem}[section]
\newtheorem{lemma}[theorem]{Lemma}
\newtheorem{proposition}[theorem]{Proposition}

\newtheorem{conjecture}[theorem]{Conjecture}
\newtheorem{main-conjecture}[theorem]{Main Conjecture}

\newtheorem{lem-def}[theorem]{Lemma and Definition}
\newtheorem{prop-def}[theorem]{Proposition and Definition}
\newtheorem{observe/question}[theorem]{Observation/Question}

\theoremstyle{definition}
\newtheorem{definition}[theorem]{Definition} 
\newtheorem{remark}[theorem]{Remark}

\newtheorem{rem-def}[theorem]{Remark and Definition}
\newtheorem{example}[theorem]{Example}
\newtheorem{question}[theorem]{Question}

\newcount\HOUR
\newcount\MINUTE
\newcount\HOURSINMINUTES
\newcount\INTVAL
\newcommand{\twodigit}[1]{\INTVAL=#1\relax\ifnum\INTVAL<10 0\fi\the\INTVAL}
\HOUR=\time\divide\HOUR by 60\relax
\HOURSINMINUTES=\HOUR\multiply\HOURSINMINUTES by 60\relax
\MINUTE=\time\advance\MINUTE by -\HOURSINMINUTES\relax
\newcommand\rightnow{
            \twodigit{\number\day}.\space
            \ifcase\month\or January\or February\or March\or April\or
May\or June\or July\or August\or September\or October\or November\or
December\fi
            \space\number\year}

\title[A Survey on The $\EGH$ conjecture ]{A survey on The Eisenbud-Green-Harris Conjecture} 

\author{Sema G\"unt\"urk\"un}
\address{Department of Mathematics and Statistics, Amherst College \\  Seeley Mudd Bldg., 31 Quadrangle Dr., Amherst, MA 01002, USA}
\email{sgunturkun@amherst.edu}


\renewcommand*{\backref}[1]{}
\renewcommand*{\backrefalt}[4]{%
 \ifcase #1 (Not cited.)%
   \or        (Cited on page~#2.)%
    \else      (Cited on pages~#2.)%
    \fi}

\begin{document}



\begin{abstract} 
The Eisenbud-Green-Harris (EGH) conjecture offers a generalization of the famous Macaulay's theorem about the Hilbert functions of homogeneous ideals in a polynomial ring $K[x_1,\ldots, x_n]$.
In this survey paper, we provide a good compilation of results on the EGH conjecture that have been obtained so far. We discuss these results in terms of their approaches. 
 \end{abstract}

\maketitle 

\section{Introduction}  

Let $I$ be a homogeneous ideal given in a standard graded polynomial ring $R$ in $n$ variables over a field $K$, and let $I_d$ denote the degree $d$ graded component of $I$. 
Assuming the $K$-dimension of $I_d$ is known, it sounds a quite simple question to ask what one can say about the dimension of the graded component of $I$ in degree $d+1$, and yet it attracts a lot of attention in commutative algebra and algebraic geometry.  An answer to this question was given by Macaulay's breakthrough work \cite{Ma} by providing a numerical bound for the growth of Hilbert function depending on the value  at the preceding degree. He showed that Hilbert functions of  special monomial ideals, called \textit{lexicographic} ideals,  describe all possible Hilbert functions of homogeneous ideals in $R$.  Macaulay's result led to other classical results on Hilbert functions such as Gotzmann's Persistence Theorem and Green's Hyperplane Restriction Theorem (see \cite{BH} for nice treatments of all these theorems). 

Generalizations of Macaulay's result on the extremal behavior of lexicographic ideals for Hilbert functions allows to relate a homogeneous ideal containing the powers of variables with a monomial ideal containing the same powers of variables (see \cite{Kr,Ka,CL}).

In their \emph{Higher Castelnuovo Theory} paper, Eisenbud, Green and Harris conjectured a further generalization of Macaulay's theorem for homogeneous ideals containing a regular sequence in certain degrees (see Conjecture \ref{EGH}).  Eisenbud-Green-Harris (EGH) conjecture, motivated by Cayley-Bacharach theorems, suggests a refinement of Macaulay's bound on the growth of the Hilbert function by involving the information of degrees of the regular sequence contained in the ideal minimally. Although there are notable works done on Eisenbud-Green-Harris (EGH) conjecture, it is still widely open after more than  25 years.

A survey on lex-plus-powers ideals by Francisco and Richert \cite{FR} provides a very good source to understand these special monomial ideals thoroughly, and it also discusses the EGH conjecture and its equivalent variations in details. Since there have been significant progress on the EGH conjecture since \cite{FR}, the main intent of our survey paper is to contribute the literature by providing the current state of the EGH conjecture and to assemble the results that have been obtained so far.  

As plan of this paper, in \textsection 2 we state some preliminaries and review Macaulay's results on Hilbert function. Section 3 lays out the Eisenbud-Green-Harris (EGH) conjecture and its variations.  In \textsection 4, we present the results obtained on the EGH conjecture by grouping them in terms of their approaches. Finally, in \textsection 5 we point out the open cases of the EGH conjecture, and we recall a closely related conjecture known as the Lex-Plus-Powers conjecture. We conclude the final section with some applications of EGH.

%
%
%
%
%
%
%
%
%
%
%
%
 \section{Preliminaries and Macaulay's theorem on Hilbert functions}  

We let $R$ be  the polynomial ring $K[x_1,\ldots, x_n]$ over a field $K$ with standard grading $ R = \oplus_{i\geq 0} R_i$ where $R_i$ is the $i$-th graded component. 

We fix the lexicographic order as $x_1 >_{\text{lex}} x_2 >_{\text{lex}} \ldots >_{\text{lex}} x_n$.  Then we define the monomial order between two monomials of the same degree as  $x_1^{a_1}x_2^{a_2}\cdots x_n^{a_n} >_{\text{lex}} x_1^{b_1}x_2^{b_2}\cdots x_n^{b_n}$  if $a_i > b_i$ where $i$ is the smallest index such that $a_i \neq b_i$.  For the sake of simplicity, we use $>$ for $>_{\text{lex}}$.

\begin{definition} Let $I$ be a monomial ideal in $R$ minimally generated by monomials $m_1,\ldots, m_k$. We call $I$ a \textit{lexicographic ideal} or simply a \textit{lex  ideal} if it satisfies the following property: for any monomial  $m'$ in $R$ with $\deg m' = \deg m_i$ and  $m'>m_i$ for some $i=1,\ldots, k$, then $m' \in I$ as well. 
\end{definition}
We next define another special type of monomial ideal in our context. 

\begin{definition} For given $2\leq a_1 \leq a_2 \leq \ldots \leq a_n$, we call a monomial ideal $L$ a \textit{lex-plus-powers ideal associated with degree $(a_1,\ldots, a_n)$} if it can be written as
\[ L = (x_1^{a_1}, \ldots, x_n^{a_n}) + J\]
where $J$ is a lex ideal in $R$. 
\end{definition}

%
%

For any homogeneous ideal $I$ we define the \textit{Hilbert function} of $R/I$ as
\[\Hf_{R/I}(t) = \dim_K (R/I)_t  = \dim_K R_t - \dim_K I_t\]
where $R_t$ is the degree $t$ graded component of $R$ with $\dim_K R_t = \binom{n-1+t}{t}$, and $I_t$ is the degree $t$ graded component of the ideal $I$.



\begin{example} In $K[x,y,z]$, consider the monomial ideal $I = (x^3, xy, y^4, yz, z^2)$. Then simple computations give us the graded components of $R/I$; 
$(R/I)_1 = \text{$K$-}\Span\{ x,y,z\}$, $(R/I)_2 = \text{$K$-}\Span\{ x^2, xz, y^2 \}$, $(R/I)_3 = \text{$K$-}\Span\{ x^2z, y^3\}$ and $dim_K(R/I)_i = 0$ for $i\geq 4$. Therefore, one expresses the Hilbert function of $R/I$ as $(h_0, h_1, h_2, h_3) = (1, 3, 3, 2)$ where $h_i =\Hf_{R/I}(i)$ and $h_i=0$ for $i\geq 4$.
\end{example}

As we see in the above example, it is possible to compute the Hilbert function of monomial ideals even by hand, however for arbitrary homogeneous ideals it becomes challenging to calculate without using a software such as Macaulay 2 (see \cite{M2} package \textit{LexIdeals}, command \verb|hilbertFunct|).

We call a sequence $f_1,\ldots, f_r$ of forms in $R = K[x_1,\ldots, x_n]$  \emph{a regular sequence} of length $r$ if, for each $i=1,\ldots,r$, $f_i$ is a nonzero divisor for the ring $R/(f_1,\ldots, f_{i-1})$.  If $r = n$, then the regular sequence has full length, in this case, it is referred as a maximal regular sequence.

\begin{remark}\label{HF_of_ci}  Let $\mathfrak{c}$ be a homogeneous ideal in $R$ generated by a regular sequence $f_1,\ldots, f_n$ with $\deg f_i = a_i$ for $i=1,\ldots, n$. We call $\mathfrak{c}$ a \textit{complete intersection ideal of type $(a_1,\ldots, a_n)$} and the ring $R/\mathfrak{c}$ is called a complete intersection ring. Then the Hilbert function of $R/\mathfrak{c}$ is $\Hf_{R/\mathfrak{c}}(i) = \Hf_{R/(x_1^{a_1},\ldots, x_n^{a_n})}(i)$ for all $i\geq 0$.

Furthermore,  $\Hf_{R/(x_1^{a_1},\ldots, x_n^{a_n})}(i) =  \Hf_{R/(x_1^{a_1},\ldots, x_n^{a_n})}(s-i)$
where $s = \sum\limits_{i}^{n}(a_i-1)$. 
\end{remark}

The following proposition provides a very useful relation between the Hilbert function of an ideal $I$ and the Hilbert function of another ideal generated by a regular sequence contained in $I$ under the liaison (see \cite[Theorem 3]{DGO}). 
\begin{proposition}\label{dual} Let $I$ be a homogeneous ideal and $\mathfrak{c} \subseteq I$ a complete intersection ideal, and $s= \sum\limits_{i=1}^n(a_i-1)$. Then, for all $j\geq 0$,
\[  \Hf_{R/I}(j) = \Hf_{R/\mathfrak{c}}(j) - \Hf_{R/(\mathfrak{c}:I)}(s-j) \]
\end{proposition}

\medskip

For a given two positive integers $d$ and $i$,  
the  \textit{$i$-th Macaulay representation of  $d$ } (also known as the $i$-th binomial expansion of $d$), denoted $d^{(a)}$, is given by $$d = \binom{m_i}{i}+\binom{m_{i-1}}{i-1} + \ldots + \binom{m_2}{2}+\binom{m_1}{1} $$ where $m_i > m_{i-1} > \ldots >  m_1 \geq 0$ are uniquely determined and are called the $i$-th Macaulay coefficients of $d$.  In this case, we let
 \[d^{\langle i \rangle} := \binom{m_i+1}{i+1}+\binom{m_{i-1}+1}{i} + \ldots + \binom{m_2+1}{3}+\binom{m_1+1}{2}.\]

To give a simple example, let $d=21$, $i=4$. Then $21^{(4)} = \binom{6}{4} + \binom{4}{3} +\binom{2}{2}+\binom{1}{1}$, therefore $21^{\langle 4 \rangle }  = \binom{7}{5} + \binom{5}{4} +\binom{3}{3} +\binom{2}{2}= 28$.

We next state Macaulay's well-known theorem on Hilbert functions.

\begin{theorem}[\cite{Ma, BH}] Let $I$ be a homogeneous ideal in the polynomial ring $R$.
\begin{itemize}
\item[(a)]  There is a lex ideal in $R$ with the same Hilbert function, and this lex ideal is uniquely determined.  
\item[(b)] [\textit{Macaulay's bound}]  If $\Hf_{R/I}(i) = d$ then 
\[ \Hf_{R/I}(i+1) \leq d^{\langle i \rangle}.\]
\end{itemize}
\end{theorem}

\begin{example} \label{MacBound-R} 
Notice that $\Hf_R(i) = \binom{n+i-1}{i} = d$ and so the $i$-th Macaulay representation of $d$ is $d^{( i )}= \binom{n-1+i}{i}$. Computing $\Hf_R(i+1) = \binom{n+i}{i+1}$ shows that $R$ attains exactly Macaulay's bound $d^{\langle i \rangle}$. 
\end{example}

The following important theorem shows when a homogeneous ideal carries a similar behavior of attaining Macaulay's bound as in Example \ref{MacBound-R}.

\begin{theorem}[Gotzmann's Persistence Theorem \cite{Gotzmann,BH}]
Let $I$ be a homogeneous ideal in $R$ generated by forms of degree $\leq d$. If the Hilbert function of $R/I$ achieves Macaulay's bound in the next degree $d+1$, that is $\Hf_{R/I}(d+1) = \Hf_{R/I}(d)^{\langle d\rangle},$
then \[ \Hf_{R/I}(j) = \Hf_{R/I}(j-1)^{\langle d\rangle} \quad \text{for all} \quad j \geq d.\]
\end{theorem}

Another classical result on the growth of Hilbert functions worth to mention is given by Green.  This result was also used to give an elegant proof of Macaulay's theorem. 

\begin{theorem}[Green's Hyperplane Restriction Theorem \cite{Gr, BH}]
Let $I$ be a homogeneous ideal in $R$, and let $t\geq 1$ be a given degree. Then 
\[ \Hf_{R/I+(\ell)}(t) \leq \binom{m_t-1}{t}+\binom{m_{t-1}-1}{t-1} + \ldots + \binom{m_2-1}{2}+\binom{m_1-1}{1}\]
where \[\Hf_{R/I}(t)^{(t)} = \binom{m_t}{t}+\binom{m_{t-1}}{t-1} + \ldots + \binom{m_2}{2}+\binom{m_1}{1}\] is the $t$-th Macaulay  representation of $\Hf_{R/I}(t)$.
\end{theorem}

\section{The EGH conjecture}

A generalization of Macaulay's results was considered by studying  homogeneous ideals in  $S=R/(x_1^{a_1},\ldots, x_n^{a_n})$ instead of homogeneous ideals in the polynomial ring $R=K[x_1,\ldots, x_n]$.  The existence of the lex ideal in $S$ with the same Hilbert function was shown by Kruskal \cite{Kr} and Katona \cite{Ka} when $a_1=a_2=\ldots=a_n = 2$ and more generally when $2\leq a_1\leq \ldots \leq a_n$ was done by Clements and Lindstr\"om \cite{CL} and also by Greene-Kleitman \cite{GrKl}. These results were obtained in set theoretical and combinatorial settings.

A question can be raised for a similar behavior for the homogeneous ideals in $R/(f_1,\ldots, f_n)$ where $f_1,\ldots, f_n$ is a regular sequence in $R$.  In \cite{EGH}, Eisenbud, Green and Harris initially stated the following conjecture for the case of  regular sequence of quadrics.

\begin{conjecture}\label{orig-EGH} Given homogeneous ideal $I$ in $R$ containing a full length regular sequence of quadratic forms.  Let $\Hf_{R/I}(i) = h$ and the $i$-th Macaulay representation of $h$ be $$h^{(i)} = \binom{m_i}{i}+\binom{m_{i-1}}{i-1} + \ldots + \binom{m_2}{2}+\binom{m_1}{1}$$ where $m_i > m_{i-1} > \ldots \geq m_1 \geq 0$. 

Then \[\Hf_{R/I}(i+1) \leq \binom{m_i}{i+1}+\binom{m_{i-1}}{i} + \ldots +\binom{m_1}{2}.\]
\end{conjecture} 

The new bound proposed by Conjecture \ref{orig-EGH} is finer than Macaulay's bound. We can see this
using the binomial identity $\binom{m+1}{k} = \binom{m}{k}+\binom{m}{k-1}$,
\begin{align*}
h^{\langle i \rangle } &= \binom{m_i+1}{i+1}+\binom{m_{i-1}+1}{i} + \ldots +\binom{m_1+1}{2}  \\
&=  \binom{m_i}{i+1} + \binom{m_i}{i} + \binom{m_{i-1}}{i}+\binom{m_{i-1}}{i-1} + \ldots + \binom{m_1}{2}+\binom{m_1}{1} \\
&= h^{(i)} + \binom{m_i}{i+1}+\binom{m_{i-1}}{i} + \ldots +\binom{m_1}{2}\\
&> \binom{m_i}{i+1}+\binom{m_{i-1}}{i} + \ldots +\binom{m_1}{2}.
\end{align*}

\begin{example} Suppose $I\subseteq K[x_1,\ldots,x_7]$ is a homogeneous ideal containing a regular sequence of quadratic forms $f_1,\ldots, f_7$ and $\Hf_{R/I}(2) = 17$. The possible growth for the Hilbert function in degree $3$ by Macaulay's bound is $\Hf_{R/I}(3)\leq 38$, but Conjecture \ref{orig-EGH} claims that $\Hf_{R/I}(3)\leq 21$.
\end{example}

If the homogeneous ideal $I$ is generated by generic quadrics, it is already known that Conjecture \ref{orig-EGH} is true by Herzog and Popescu \cite{HP} when the characteristic is zero. When $K$ has arbitrary characteristic, this was shown by Gasharov \cite{Gasharov}.

The main motivation behind Conjecture \ref{orig-EGH} about homogeneous ideals containing quadratic regular sequence was another conjecture, known as the Generalized Cayley-Bacharach conjecture, stated in \cite{EGH} in more geometric perspective. 

\begin{conjecture}\label{GCB}[Generalized Cayley-Bacharach Conjecture for quadrics]
Let $\Gamma$ be a complete intersection of $n$ quadrics in $\mathbb{P}^n$.  Any hypersurface $X \subset \mathbb{P}^n$ of degree $d$ containing a subscheme $\Omega\subset \Gamma$ of degree strictly greater than $2^n - 2^{n-d}$ must contain $\Gamma$.
\end{conjecture}
Conjecture \ref{orig-EGH} implies the Generalized Cayley-Bacharach conjecture for quadrics. 

In the same article \cite{EGH}, Eisenbud, Green and Harris dropped the quadratic condition on the regular sequence, and further conjectured the same statement for homogeneous ideals containing regular sequences with any degrees $2\leq a_1 \leq a_2 \ldots \leq a_n$.

\begin{conjecture}[Eisenbud-Green-Harris (EGH) Conjecture, \cite{EGH}] \label{EGH} Let $I$ be a homogeneous ideal in $R = K[x_1,\ldots, x_n]$ containing a regular sequence $f_1,\ldots, f_n$ with degrees $a_1,\ldots, a_n$ such that $2\leq a_1\leq \ldots \leq a_n$. Then there is a lex-plus-powers ideal $L = (x_1^{a_1}, \ldots, x_n^{a_n}) + J$ with a lex ideal  $J$ in $R$ such that
\[\Hf_{R/I}(i) = \Hf_{R/L}(i) \quad \text{for all} \quad i\geq 0.\]
\end{conjecture}

From now on, we will refer to this conjecture as the EGH conjecture. We will also use $\EGH_{(a_1,\ldots, a_n), n}$ to emphasize the degrees of the regular sequence and also that it is a full length-$n$ regular sequence in $R=K[x_1,\ldots, x_n]$. 

Notice that in Remark \ref{HF_of_ci}, we observed a very trivial version of this statement when $I =(f_1,\ldots, f_n)$ and $L=(x_1^{a_1},\ldots, x_n^{a_n})$.

Another statement of the Generalized Cayley-Bacharach conjecture that does not require quadrics was given in \cite[Conjecture CB12]{EGH2}.  In 2013, Geramita and Kreuzer reformulated this version of Generalized Cayley-Bacharach conjecture for arbitrary degrees by dividing it into intervals. They also strengthened the Conjecture CB12 in \cite[Conjecture 3.5]{GK}. In $\mathbb{P}^3$, they provided a proof for it. In $\mathbb{P}^n$, they confirmed \cite[Conjecture 3.5]{GK} for some intervals. The EGH conjecture which is the concern of this paper is stronger than \cite[Conjecture CB12]{EGH2} as well. 

One of the variations of the EGH conjecture in the literature is when one allows to have a regular sequence that is not of full length. 

\begin{conjecture}[$\EGH_{n, (a_1,\ldots, a_r), r}$] \label{non-max_regseq} Let $I$ be a homogeneous ideal in $R=K[x_1,\ldots, x_n]$ containing a regular sequence of length $r<n$ with degrees $a_1,\ldots, a_r$ such that $2\leq a_1\leq \ldots \leq a_r$. Then there is a lex-plus-powers ideal $L = (x_1^{a_1}, \ldots, x_r^{a_r}) + J^{\text{lex}}$ in $R$ with the same Hilbert function as $I$.
\end{conjecture}

\begin{remark}\label{Rmk3.5} The equivalence between Conjectures \ref{EGH} and \ref{non-max_regseq} was discussed by Caviglia and Maclagan \cite{CM}.  For $2\leq a_1 \leq \ldots \leq a_r$ fixed, they showed that  if $\EGH_{(a_1',\ldots, a_n'), n}$ holds for all $2\leq a_1'\leq \ldots \leq a_n'$ where $a_i' = a_i$ for $i=1,\ldots, r$ then $\EGH_{n, (a_1,\ldots, a_r), r}$ holds (see \cite[Propositions 9]{CM}). They also showed that if $\EGH_{(a_1,\ldots, a_r), r}$ holds then $\EGH_{n, (a_1,\ldots, a_r), r}$ holds for all $n\geq r$ (see \cite[Propositions 10]{CM}).
\end{remark}
%
%
%
%
%
%
%
%
%
%
%
\section{Results on the EGH conjecture}

Richert  \cite{Ri} proved that the EGH conjecture is true for $R=K[x, y]$.  Thus, for any homogeneous ideal $I$ in two variables containing a regular sequence $f_1, f_2$ with degrees $2\leq a_1\leq a_2$,  we have  a lex plus powers ideal $L=(x^{a_1}, y^{a_2})+J$ such that 
\[\dim_K I_i = \dim_K L_i \quad \text{ for all } \quad i\geq 0.\]

For $K[x_1,\ldots, x_n]$ with  $n>2$, we put together the known results on EGH  depending on the approaches were used.

\medskip

\subsection{EGH depending on the degrees $(a_1,\ldots, a_n)$.}
Let $n\geq 2$. For a fixed degree $d\geq 1$, when the Hilbert function of $R/I$ at degree $d$ is known,
the EGH conjecture proposes a maximal growth for degree $d+1$.  One of the adopted approaches in the literature focuses on the growth at certain degree.  Hence, the following definition states a partial version of EGH conjecture for consecutive degrees.

\begin{definition}[$\EGH_{(a_1,\ldots, a_n), n}(d)$] \label{weakerEGH} For any homogeneous ideal $I$ in $R$ containing a regular sequence $f_1,\ldots, f_n$ of degrees $2\leq a_1\leq \ldots \leq a_n$ respectively, if there exists a lex-plus-powers ideal $L$ associated with degrees $a_1,\ldots, a_n$ such that 
\[\Hf_{R/I}(d) = \Hf_{R/L}(d) \quad \text{ and } \quad \Hf_{R/I}(d+1) = \Hf_{R/L}(d+1),\]
 then we say that $\EGH_{(a_1,\ldots, a_n), n}(d)$ holds.
\end{definition}

The following proposition is given by Francisco \cite{Fr} for the almost complete intersection ideals.

\begin{proposition}\label{Fr-prop} Let $I = (f_1,\ldots, f_n, g)$ be a homogeneous ideal where $f_1, \ldots, f_n$  is a regular sequence with degrees $a_1,\ldots, a_n$ and $\deg g =d\geq a_1$. Then $\EGH_{(a_1,\ldots, a_n), n}(d)$ is true for $I$.
\end{proposition}

In \cite{CDS2}, for an almost complete intersection $I = (f_1,f_2, f_3, g) \subset K[x_1,\ldots, x_n]$ where $f_1, f_2, f_3$ is a regular sequence of length three with $\deg f_i = a_i$, $i=1,2,3$ and $\deg g = d \leq a_1+a_2+a_3-3$,  Caviglia-De Stefani showed $\Hf_{R/I}(i) \leq \Hf_{R/L}(i)$ for all $i\geq 0$ where $L = (x_1^{a_1}, x_2^{a_2}, x_3^{a_3}, \mathrm{m})$ with $\mathrm{m}$ is the largest monomial of degree $d$ with respect to lexicographic order that is not in $(x_1^{a_1}, x_2^{a_2}, x_3^{a_3})$. Their work on such almost complete intersections also recovered the result of \cite{GK} for $\mathbb{P}^3$.

To show that $\EGH_{(d,\ldots, d), n}(d)$ holds  it suffices to show that the statement in Definition \ref{weakerEGH} holds for the homogeneous ideals $I = (f_1,\ldots, f_n, g_1,\ldots, g_m)\subseteq K[x_1,\ldots, x_n]$ generated by degree $d$ forms $f_1,\ldots, f_n, g_1,\ldots, g_m$ where $f_1,\ldots, f_n$ form a regular sequence (see \cite[Lemma 2.6]{GH}.)  In other words, it is enough to show the statement for the ideals where not only the regular sequence $f_1,\ldots, f_n$ in the generators have degree $d$, but also rest of the generators $g_1,\ldots, g_m$ have degree $d$ too. 
Thus, focusing on the case $d=2$ we have the following remark.

\begin{remark}\label{rmk-quadrics} In order to show that $\EGH_{(2,\ldots, 2), n}(2)$ is true, it suffices to study the ideals generated by only quadrics containing a maximal regular sequence. 
\end{remark}

In \cite{CM}, Caviglia-Maclagan provided the following lemma about this weaker version of  the EGH. Due to its importance as a tool for studying the EGH conjecture, we would like to present its proof given in \cite[Lemma 12]{CM}. 

\begin{lemma}\label{lemma} Given $2\leq a_1 \leq \ldots \leq a_n$, set $s= \sum\limits_{i=1}^n (a_i-1)$. Let $d\geq 1$. Then 
$$ \mbox{
$\EGH_{(a_1, \ldots, a_n), n}(d)$ holds if and only if $\EGH_{(a_1, \ldots, a_n), n}(s-d-1)$ holds.}$$
Furthermore,
$$\mbox{
if $\EGH_{(a_1, \ldots, a_n), n}(d)$ holds for all $0\leq d \leq \floor{\frac{s-1}{2}}$ then  $\EGH_{(a_1, \ldots, a_n), n}$ holds.}$$
\end{lemma}

\begin{proof} Suppose that $\EGH_{(a_1, \ldots, a_n), n}(d)$ holds. Then given any homogeneous ideal containing a regular sequence with degrees $a_1,\ldots, a_n$, there is lex-plus-powers ideal associated with degrees $a_1,\ldots, a_n$ such that Hilbert functions of both ideals agree at degrees $d$ and $d+1$.
Let $I$ be a homogeneous ideal in $R$ containing a regular sequence $f_1,\ldots, f_n$ with $\deg f_i=a_i$, for $i=1,\ldots, n$.  Then by Proposition \ref{dual} we get
\begin{align*} 
\Hf_{R/I}(j) = \Hf_{R/(f_1,\ldots,f_n)}(j) - \Hf_{R/((f_1,\ldots,f_n) : I)}(s-j).
\end{align*}

Since the colon ideal $((f_1,\ldots,f_n) : I)$ contains the regular sequence $f_1,\ldots, f_n$, then by assumption there is a lex-plus-powers ideal $L' = (x_1^{a_1},\ldots, x_n^{a_n}) + J'$ such that 
\begin{align}\label{eqn2}
\Hf_{R/((f_1,\ldots,f_n): I)}(d) =\Hf_{R/L'}(d)      \quad \text{ and } \quad \Hf_{R/((f_1,\ldots,f_n) : I)}(d+1) = \Hf_{R/L'}(d+1).  
\end{align}
On the other hand, we also have 
\begin{align}\label{eqn3}
\Hf_{R/L'}(j) = \Hf_{R/(x_1^{a_1},\ldots,x_n^{a_n})}(j) - \Hf_{R/((x_1^{a_1},\ldots,x_n^{a_n}) : L')}(s-j).
\end{align}

Thus, for $j=s-d-1$ and $j=s-d$, 
\begin{align*}
\Hf_{R/I}(j) &= \Hf_{R/(f_1,\ldots,f_n)}(j) - \Hf_{R/((f_1,\ldots,f_n): I)}(s-j)  \\     &=  \Hf_{R/(f_1,\ldots,f_n)}(j) - \Hf_{R/(x_1^{a_1},\ldots,x_n^{a_n})}(s-j) + \Hf_{R/((x_1^{a_1},\ldots,x_n^{a_n}) : L')}(j) \\
    &=  \Hf_{R/((x_1^{a_1},\ldots,x_n^{a_n}) : L')}(j).
\end{align*}
The second equality follows from \eqref{eqn2} and \eqref{eqn3}. Then the last equality is by Remark \ref{HF_of_ci} and Proposition \ref{dual}.

Finally, since the ideal $((x_1^{a_1},\ldots,x_n^{a_n}) : L')$ contains the regular sequence $x_1^{a_1},\ldots,x_n^{a_n}$, by Clement-Lindstr\"om result mentioned previously, there exists a lex-plus-powers ideal $L$ associated with degrees $a_1,\ldots, a_n$ such that $\Hf_{R/((x_1^{a_1},\ldots,x_n^{a_n}) : L')}(i) = \Hf_{R/L}(i)$ for all $ i\geq 0$.

Hence, 
$$\Hf_{R/I}(s-d-1) = \Hf_{R/L}(s-d-1) \quad \text{ and } \quad \Hf_{R/I}(s-d) = \Hf_{R/L}(s-d).$$

\end{proof}

\begin{remark}\label{EGH(0)}  We know $\Hf_{R/I}(0)=1$.  If $\Hf_{R/I}(1)=n$, then any lex-plus-powers ideal $L = (x_1^{a_1},\ldots, x_n^{a_n}) +J$ where the lex ideal $J$ does not contain any linear form  has $\Hf_{R/L}(0)=1$ and $\Hf_{R/L}(1)=n$. If $\Hf_{R/I}(1)=r<n$, that is $I$ has $n-r$ linear generators, then it is enough to pick the lex ideal $J$ containing $x_1,\ldots, x_{n-r}$, then $\Hf_{R/L}(1)=r$ as well. Therefore, we see that $\EGH_{(a_1,a_2,\ldots, a_n), n}(0)$ is always true.
\end{remark}
%
%
%
%

\begin{theorem}[\cite{CM,CDS}] For  $2\leq a_1\leq \ldots \leq a_n$  such that $a_j \geq \sum\limits_{i=1}^{j-1}(a_i-1)$, the $\EGH_{(a_1,\ldots,a_n), n}$ is true. 
\end{theorem}

The strict inequality was shown by Caviglia-Maclagan in \cite{CM}. Their proof used an inductive argument on $n$, and followed from Lemma \ref{lemma} and Remark \ref{Rmk3.5}. Very recently, Caviglia-De Stefani \cite{CDS} extended this degree growth condition by including the equality. Their work actually provided a stronger case.  They showed that if a homogeneous ideal $I \subseteq R$ contains a regular sequence $f_1,\ldots, f_{n-1}$ with degrees $2\leq a_1 \leq \ldots \leq a_{n-1}$ satisfies the $\EGH_{n, (a_1,\ldots, a_{n-1}), n-1}$, then $I+(f_n)$ satisfies the EGH conjecture for any $f_n$ where $f_1,\ldots, f_n$ form a regular sequence and $\deg f_n \geq \sum\limits_{i=1}^{n-1}(a_i-1)$. (see \cite[Theorem 3.6]{CDS}).

%


The result of Caviglia-Maclagan with the recent improvement by Caviglia-De Stefani provides an affirmative answer for the EGH conjecture for a large case in terms of the degrees $a_1,\ldots, a_n$.  One of the interesting case that is not covered by this result is when $a_1=a_2=\ldots = a_n$, more specifically as in Conjecture \ref{orig-EGH} when $n\geq 4$.  We will focus on the quadratic case $a_i=2$ for all $i=1,\ldots, n$ separately (see Subsection \ref{subsec-quadrics}).

In \cite{Cooper}, Cooper approached the EGH conjecture for $n=3$ in a geometric setting by investigating the Hilbert functions of the subsets of complete intersections in $\mathbb{P}^2$ and $\mathbb{P}^3$. She showed that the $\EGH_{(a_1,a_2, a_3), 3}$ is true for the degrees $(2, a, a)$ for $a\geq 2$ and $(3, a, a)$ for $a\geq 3$.

Another result for $n=3$ for the Gorenstein ideals containing a regular sequence with degrees $2\leq a_1 \leq a_2 \leq a_3$ was proven by Chong \cite{Chong}.

\medskip

\subsection{EGH via liaison.} 
Chong's work covers EGH beyond Gorenstein ideals in $K[x,y,z]$. It uses the linkage theory and studies a special subclass of licci ideals.  First, we would like to review some definitions and concepts related to linkage theory for ideals in $R=K[x_1,\ldots, x_n]$ to present Chong's result.

Let $I, I' \subseteq R$ be homogeneous ideals of height $r\leq n$.  If there exists a regular sequence $g_1,\ldots, g_r$ such that the complete intersection $\mathfrak{c} = (g_1,\ldots, g_r) \subseteq I\cap I'$, and $I = \mathfrak{c} : I'$ and $I' = \mathfrak{c} : I$, then we say that  $I$ and $I'$ are \textit{linked (algebraically) via $\mathfrak{c}$}. We express this linkage as $I\linkby[\mathfrak{c}] I'$.   If $I$ minimally contains a regular sequence $f_1,\ldots, f_r$ with degrees $2\leq a_1\leq \ldots \leq a_r$ and if the link $\mathfrak{c}$ is a complete intersection of type $(a_1,\ldots, a_r)$ then we say $\mathfrak{c}$ is a \textit{minimal link}.

Suppose that there is a finite sequence of links $I \linkby[\mathfrak{c}_1] I_1 \linkby[\mathfrak{c}_2] \cdots \linkby[\mathfrak{c}_t]I_t$ where $I_t$ is a complete intersection, we say that $I$ is in the linkage class (a.k.a. liaison class) of the complete intersection $I_t$.  An ideal in the linkage class of a complete intersection is called \textit{licci}.  We next state the work of Chong on this.

\begin{theorem}\cite{Chong} Let $2\leq a_1\leq \ldots \leq a_n$, and $I\subseteq R$ be a a homogeneous ideal containing a regular sequence of degrees $a_1,\ldots, a_n$.  
If $I$ is licci such that $I \linkby[\mathfrak{c}_1] I_1 \linkby[\mathfrak{c}_2] \cdots \linkby[\mathfrak{c}_t]I_t$ where each link $\mathfrak{c}_i$ has type $\bar{a}_i$ for $i=1,\ldots, n$ with $\bar{a}_1\geq \bar{a}_2 \geq \ldots \geq \bar{a}_n$, and $\mathfrak{c}_1$ is a minimal link, that is $\bar{a}_1 = (a_1,\ldots, a_n)$, then there is a lex-plus-powers ideal associated with degrees $a_1,\ldots, a_n$ with the same Hilbert function as $I$.
\end{theorem}

In the same paper \cite{Chong}, Chong also proved that $\EGH_{n, (a_1,\ldots, a_r), r}$ holds for the licci ideals where the types of the links satisfy the ascending condition and the first link is minimal. His result on Gorenstein ideals when $n=3$ we mentioned in previous subsection is a consequence of this result.
\medskip

\subsection{EGH via the structure of the regular sequence.}

Let $I \subseteq R=K[x_1,\ldots, x_n]$ be a homogeneous ideal containing a regular sequence $f_1,\ldots, f_n$ with degree $2\leq a_1 \leq \ldots \leq a_n$, respectively. By Clements-Lindstr\"om's result, we already know that $\EGH_{(a_1,\ldots,a_n), n}$ is true when $f_i = x_i^{a_i}$ for all $i=1,\ldots, n$.

Mermin and Murai \cite{MM} proved a special case of $\EGH_{n, (a_1,\ldots, a_r), r}$, $r<n$, when $\Char K =0$. For the homogeneous ideals containing a regular sequence $f_1,\ldots, f_r$ formed by monomials with degrees $2\leq a_1 \leq \ldots \leq a_r$, they showed the existence of lex-plus-powers ideal associated with $(a_1,\ldots, a_r)$ with the same Hilbert function.

Another notable work regarding the structure of the regular sequence contained in the ideal is done by Abedelfatah in \cite{Abed}. He showed that if a homogeneous ideal $I$ containing a regular sequence $f_1,\ldots, f_n$ such that $\deg f_i =a_i$, $i=1,\ldots,n$ and each $f_i$ splits into linear factors completely, then $I$ has the same Hilbert function of a lex-plus-powers ideal $(x_1^{a_1},\ldots, x_n^{a_n})+J$  in $R$.

Shortly after, Abedelfatah extended this result in \cite{Abed2}.

\begin{theorem} Let $\mathfrak{a}$ be an ideal generated by the product of linear forms and contain a regular sequence $f_1,\ldots, f_r$ with degrees $a_1,\ldots, a_r$. Let $I$ be a homogeneous ideal in $R$ such that $(f_1\ldots, f_r) \subset \mathfrak{a} \subset I$ then the Hilbert function of $I$ is the same as the lex-plus-powers ideal containing powers $x_1^{a_1},\ldots,x_r^{a_r}$.
\end{theorem} 
 
The previous result in \cite{Abed} is simply the case when $\mathfrak{a} = (f_1,\ldots, f_r)$.  Let $I \subset R$ be a height $r$ monomial ideal containing a regular sequence of degrees $a_1\leq \ldots \leq a_r$, then this theorem of Abedelfatah confirms that $I$ has the same Hilbert functions as a lex-plus-powers containing $x_1^{a_1},\ldots, x_r^{a_r}$. This also improves another related result given by Caviglia-Constantinescu-Varbaro in \cite{CCV} for height $r$ monomial ideals generated by quadrics.
\medskip

\subsection{When $\mathbf{a_1=\ldots = a_n=2}$.}\label{subsec-quadrics}

Finally we focus on the case when the regular sequence is formed by quadrics as originally conjectured by the Eisendbud-Green-Harris as in Conjecture \ref{orig-EGH}. For simplicity, we refer it as 
$\EGH_{(2,2,\ldots, 2), n} = \EGH_{{\bf \bar 2}, n}$.

We have already mentioned the cases when $n=2$ by Richert \cite{Ri} as his result shows EGH for any degree when $n=2$. Moreover, in $K[x,y,z]$, we have seen that the EGH conjecture for the degrees $(2,2,2)$ was covered by Cooper \cite{Cooper} and Caviglia-De Stefani \cite{CDS} separately, and quadratic monomial ideal case by \cite{CCV}.

In terms of the weaker version of the EGH conjecture for consecutive degrees given in Definition \ref{weakerEGH}, using Proposition \ref{Fr-prop} given by Francisco,  $\EGH_{{\bf \bar 2}, n}(2)$ is true for almost complete intersections $I = (f_1,\ldots, f_n, g)$ where $\deg f_i =2$ for all $i=1,\ldots,n$ and $\deg g=2$.  More precisely,  $\Hf_{R/I}(3) \leq \Hf_{R/L}(3)$ where $L $ is the lex-plus-powers ideal containing the squares of the variables.

An analogous result on  $\EGH_{{\bf \bar 2}, n}(2)$ for the ideals generated by a quadratic regular sequence plus two more generators is given in the following theorem. 

\begin{theorem}\cite{GH}\label{GH}  For $n\geq 5$, $\EGH_{{\bf \bar 2}, n}(2)$ holds for homogeneous ideals minimally generated by a regular sequence of quadrics and two more generators whose degrees are at least $2$. 
\end{theorem} 
Thanks to \cite[Lemma 2.6]{GH}, which is mentioned in Remark \ref{rmk-quadrics}, to prove Theorem \ref{GH} it was enough to show the statement for an ideal $I$ generated by $n+2$ quadrics containing a maximal regular sequence.  More precisely, it sufficed to show the Hilbert function of the lex-plus-powers ideal $(x_1^2,\ldots, x_n^2, x_1x_2, x_1x_3)$ in degree $3$ is greater than or equal to $\Hf_{R/I}(3)$.  This was shown by analyzing the linear relations among the generators of the ideal $I$.

\vspace*{1ex}

For a homogeneous ideal $I = (f_1,\ldots, f_n, g_1,\ldots, g_m)$ containing quadratic regular sequence $f_1,\ldots, f_n$, it is easy to see that if each $g_i$ has degree $>2$ then they don't contribute the dimension in degree $2$ and $\dim_K I_2= n$.  Therefore,  any lex plus power ideal $L=(x_1^2,\ldots, x_n^2)+J$ where $J$ is generated by monomials of degree $> 2$ gives $\dim_KL_2=n$ as well.  


We finish this section by presenting the known cases of the original EGH conjecture for $n\geq 4$.

\begin{theorem} $\EGH_{{\bf \bar 2}, n}$ is true when 
\begin{itemize}
\item[(a)] $n=4$ by Chen \cite{Chen},
\item[(b)] $n=5$ by the author and Hochster \cite{GH}.
\end{itemize}
\end{theorem}

\begin{proof}
Notice that when $a_1=\ldots = a_n=2$ we get $s= \sum\limits_{i=1}^n (a_i-1) =  n$. Then by Lemma \ref{lemma}, we get $\EGH_{{\bf \bar 2}, n}(d)$ holds if and only if $\EGH_{{\bf \bar 2}, n}(n-1-d)$ holds. 

Using this symmetry, when $n=4$, by Remark \ref{EGH(0)} we trivially have $\EGH_{{\bf \bar 2}, 4}(0)$, therefore we have $\EGH_{{\bf \bar 2}, 4}(3)$. It is enough to show $\EGH_{{\bf \bar 2}, 4}(1)$ and therefore we also get $\EGH_{{\bf \bar 2}, 4}(2)$. 

Similarly, when $n=5$, we know  $\EGH_{{\bf \bar 2}, 5}(0)$ holds, so does $\EGH_{{\bf \bar 2}, 5}(4)$. Then we need to show only  $\EGH_{{\bf \bar 2}, 5}(1)$ and $\EGH_{{\bf \bar 2}, 5}(2)$ because $\EGH_{{\bf \bar 2}, 5}(1)$ implies $\EGH_{{\bf \bar 2}, 5}(3)$.

By \cite[Proposition 2.1]{Chen},  we know that $\EGH_{{\bf \bar 2}, n}(1)$ holds for any $n\geq 2$.
Thus this completes the proof of (a).

On the other hand, the proof of (b) is done as well since $\EGH_{{\bf \bar 2}, 5}(2)$ is true as a result of Theorem \ref{GH} and Remark \ref{rmk-quadrics}.
\end{proof}

\section{Open cases of EGH and more connections.}

Although there has been a significant progress on the EGH conjecture, it is fair to say that the conjecture is still broadly open.  In this section, we discuss the open cases, and also state another well known conjecture related to the EGH conjecture. 

Besides Richert's result on EGH when $n=2$ in \cite{Ri}, we still do not know if the EGH conjecture is true when $n\geq 3$ without assuming any conditions on the degrees or on the homogeneous ideal. 

\begin{question} 
 Is $\EGH_{(a_1,a_2,a_3), 3}$ true for any given degrees $2\leq a_1 \leq a_2 \leq a_3$?
\end{question}

 We have seen that the works by Cooper, Caviglia-Maclagan and Caviglia-De Stefani cover many cases of  $(a_1,a_2,a_3)$ already. On the other hand, Chong's and Abedelfatah's results require certain conditions on the homogeneous ideals. As a result,  we can conclude that it is not known if $\EGH_{(a_1,a_2,a_3), 3}$ is true for the ideals in the following scenario: Let $I$ be a homogeneous ideal containing a regular sequence  $f_1, f_2, f_3$ such that
 \begin{itemize}
 \item the degrees $\deg f_i =a_i$, $i=1,2,3$ satisfy $ 4 \leq a_1\leq a_2 \leq a_3$ and  $a_2 \leq a_3 < a_1+a_2-2$,
 \item $I$ is not a Gorenstein ideal, and
  \item $f_1, f_2, f_3$ cannot be split into linear factors. 
 \end{itemize}

For example, $\EGH_{(4,4,5), 3}$ and $\EGH_{(4,5,5), 3}$ are two open cases with small degrees.

\begin{remark}
If we focus on the original EGH conjecture, that is, when $a_1=\ldots = a_n = 2$, $ \EGH_{{\bf \bar 2}, n}$ is still open when $n\geq6$.
\end{remark}

For a given homogeneous ideal $I$ in $R=K[x_1,\ldots, x_n]$, the graded Betti number of $I$ is $\beta_{i,j}(R/I) = \dim_K (\Tor^R_i(R/I,K))_j$.
Another well-known conjecture motivated by the EGH conjecture is given by Evans and Charalambous if these graded Betti numbers are also concerned (see lex-plus-powers ideals survey \cite{FR}).  This conjecture can be also considered analogous to Bigatti-Hulett-Pardue Theorem \cite{Bi,Hu,Pa} which is a generalization of the Macaulay's theorem for the graded Betti numbers, more precisely, it shows the extremal behavior of lex ideals for the graded Betti numbers among the homogeneous ideals with the same Hilbert function. 

\begin{conjecture}[Lex-Plus-Powers (LPP) conjecture] 
Let  $I$ be a homogeneous ideal containing a regular
sequence of degrees $2\leq a_1\leq \ldots \leq a_n$ in $R=K[x_1,\ldots, x_n]$.  Suppose that there exists a lex-plus-powers ideal $L$ with the same Hilbert function as $I$. Then the graded Betti numbers of $R/I$ cannot exceed those of $R/L$.  
That is, 
$$\mbox{$\beta_{i,j}(R/I) \leq \beta_{i,j}(R/L)$ for all $i$ and $j$.}$$ 
\end{conjecture}

Just like the EGH conjecture, the LPP conjecture remains widely open.  Nevertheless, there have been remarkable results obtained.  In \cite{Ri}, Richert showed the equivalence of the EGH conjecture and the LPP conjecture when $n=2,3$. Therefore, the LPP conjectures holds when $n=2$.  

Similar to Conjecture \ref{non-max_regseq} where ideal contains a regular sequence of length $\leq n$, one can restate the LPP conjecture by allowing non-maximal regular sequences $f_1,\ldots, f_r$ with $r<n$.  Caviglia and Kummini \cite{CK} showed that this case can be also reduced to Artinian case like the EGH conjecture.  In \cite{MM}, the LPP conjecture is shown to be true when the homogeneous ideal containing monomial regular sequence. Thus, when the regular sequence has full length then the LPP conjecture is true for the homogeneous ideals containing $x_1^{a_1},\ldots, x_n^{a_n}$. 

 In \cite{CS}, when characteristic of $K$ is 0, it was shown that the LPP conjecture holds for the homogeneous ideals containing a regular sequence with degrees $2\leq a_1 \leq \ldots \leq a_n$ if $a_i \geq\sum\limits_{ j=1}^{i-1} (a_j - 1) + 1$ for $i\geq 3$.

When $i=1$, the numbers $\beta_{1,j}(R/I)$ tells us how many generators the homogeneous ideal $I$ has in degree $j$. 
For a given homogeneous $I$ containing a regular sequence with degrees $0\leq a_1\leq \ldots \leq a_n$, the EGH conjecture claims the existence of the lex-plus-powers ideal $L$ associated with degrees $0\leq a_1\leq \ldots \leq a_n$ such that $\dim_K I_j = \dim_K L_j$ for all $j\geq 0$. Then it is well-known that this implies $\beta_{1,j}(R/L) \geq \beta_{1,j}(R/I)$. Thus, the EGH conjecture covers this particular case of the LPP conjecture when $i=1$. 

\begin{remark} 
Let  $I=(f_1,\ldots, f_n, g_1, g_2)\subseteq K[x_1,\ldots, x_n]$ be a homogeneous quadratic ideal where $f_1,\ldots, f_n$ is a regular sequence. Then, by Theorem \ref{GH}, we see that $\EGH_{{\bf \bar 2}, n}(2)$ holds for such quadratic ideals. Therefore, we get  $\dim_K I_3 \geq n^2+2n -5 = \dim_KL_3,$ where $L$ is the lex-plus-powers ideal $(x_1^2,\ldots, x_n^2, x_1x_2, x_1x_3)$.  This shows us that the number of the independent linear relations among the generators $f_1, \ldots, f_n$, $g_1$, $g_2$ is always at most $5$.
In other words, we obtain $\beta_{2,3}(R/I) \leq 5 = \beta_{2,3}(R/L)$ as well. 
\end{remark}

Richert and Sabourin, in \cite{RS}, showed that the following conjecture, a special case of the LPP conjecture when $i=n$, is equivalent the EGH conjecture. 

\begin{conjecture} Let $I$ be a homogeneous ideal in $R$ containing a regular sequence of degrees $2\leq a_1 \leq \ldots \leq a_n$ and let $L$ be a lex-plus-powers ideal associated with degrees $a_i$ such that $\Hf_{R/I}(i) = \Hf_{R/L}(i)$ for all $i\geq 0$. Then the dimension of the socle of $I$ is at most the dimension of socle of $L$ in every degree. In other words, 
$\beta_{n,j}(R/L) \geq \beta_{n,j}(R/I)$ for all $j\geq 0$.
\end{conjecture}

The LPP conjecture seems much harder than the EGH conjecture due to its strong claim on every graded Betti numbers, yet focusing on certain Betti numbers as its special cases opens up many new directions to work on.

Finally, there are some interesting applications of the EGH conjecture worth to mention.
In \cite{HWW}, Harima-Wachi-Watanabe show that, assuming the EGH conjecture is true, every graded complete intersection has the Sperner property, which simply says for a graded complete intersection $A = R/(f_1,\ldots, f_n)$, 
$\max\{ \mu(I)\ | \ I \text{ is an ideal in } A\} = \max\{\dim_K A_i  \ | \ i=0,1,2,\ldots \}$, where $\mu(I)$ is the number of minimal generators of $I$.  It is also known that the Sperner property is related to the so-called Weak Lefschetz property.

Due to geometric background of the EGH conjecture as a result of its connection to Cayley-Bacharach theory, EGH has applications in more geometric settings as well.  For instance, a recent work by Jorgenson \cite{Jor} points out that an affirmative answer for EGH has an implication on the sequence of secant indices of Veronese varieties of $\mathbb{P}^n$ (see Question 3.2 in \cite{Jor}.)

Another application of the EGH conjecture coincides with a very famous problem on decomposing real polynomials in $n$ variables as a sum of squares of real polynomials.  The cone of real polynomials that can be decomposed as a sum of squares of real polynomials is simply referred as SOS cone. In a recent work by Laplagne and Valdettaro\cite{LV}, they show that, when EGH holds, for a strictly positive polynomial on the boundary of the SOS cone, they provide bounds for the maximum number of polynomials that can appear in a SOS decomposition and the maximum rank of the matrices in the Gram spectrahedron.

\section*{Acknowledgements} The author thanks Mel Hochster for introducing and proposing to work on the EGH conjecture during her postdoctoral research. She is deeply grateful for all of their conversations.  The author thanks the referee for their valuable feedback and comments.  She also thanks Martin Kreuzer for pointing out their work.
\bigskip

%
%
%
%
%
%
%
%
%
%

\end{document}